\documentclass[preprint,12pt,3p,numbers,sort,compress]{elsarticle}

\providecommand{\pgfsyspdfmark}[3]{}
\usepackage{pgfplots}
\pgfplotsset{width=14cm,compat=1.9}
\usepackage{lineno,hyperref}
\modulolinenumbers[5]

\usepackage{amsfonts, amsthm, amssymb, amsmath}	
\usepackage{color}
\usepackage{graphicx}
\usepackage{mathrsfs}
\usepackage[english]{babel}
\usepackage{ upgreek }
\usepackage{ stmaryrd }
\usepackage{pgf,tikz}
\usetikzlibrary{arrows}

\usepackage{nicematrix}

\usepackage{nicefrac}

\usepackage{mathtools}

\journal{European Journal of Combinatorics}

\newtheorem{thm}{Theorem}[section]

\newtheorem{Theorem}[thm]{Theorem}
\newtheorem{Corollary}[thm]{Corollary}
\newtheorem{Proposition}[thm]{Proposition}
\newtheorem{Lemma}[thm]{Lemma}
\newtheorem{Example}[thm]{Example}

\newtheorem*{Theorem*}{Теорема}
\newtheorem*{Conjecture*}{Conjecture}
\newtheorem*{Corollary*}{Corollary}
\theoremstyle{definition}
\newtheorem{Definition}[thm]{Definition}

\newtheorem{Remark}[thm]{Remark}

\theoremstyle{definition}

\def\F{{\mathbb F}}
\def\A{{\cal A}}
\def\L{{\cal L}}

\def\SS{{\cal S}}
\newcommand{\FF}{\mathbb{F}}

\newcommand{\QQ}{\mathbb{Q}}
\newcommand{\ZZ}{\mathbb{Z}}
\newcommand{\NN}{\mathbb{N}}

\newcommand{\LIW}[1]{{\rm LW}_\prec({#1})}

 \newcounter{lst}

\begin{document}

\begin{frontmatter}
\title{Combinatorics on finite words and the length of a finite-dimensional associative algebra}


\author[HSE,MSU]{\texorpdfstring{M.A. Khrystik\corref{mycorrespondingauthor}}{}}
\cortext[mycorrespondingauthor]{Corresponding author}
\ead{good_michael@mail.ru}

\address[HSE]{HSE University, Faculty of Computer Science, Moscow, 101000, Russia.}
\address[MSU]{Moscow Center of Fundamental and Applied Mathematics, Moscow, 119991, Russia}

\begin{abstract}
Let $f_W(n)$ be the number of different factors of length $n$ appearing in $W$. A classical result of Morse and Hedlund, stated in 1938, asserts that an infinite word $W$ is ultimately periodic if and only if $f_W(n)\leq n$ for some $n\in \NN$. In this paper, we describe the form of finite words that satisfy the condition $f_W(n)\leq n$. We study relations between power avoidance and subword complexity of a finite word. We apply our combinatorial results to study the interrelations between various numerical invariants of finite-dimensional associative algebras.

\end{abstract}


\end{frontmatter}

\section{Introduction} \label{section:introduction}

Intuitively, it is clear that the longer a word is, the more distinct factors (subwords) it contains. It is also evident that the lesser patterns a word avoids, the more distinct factors it possesses. In this work, we prove an inequality that quantifies these simple observations. To achieve this, we require a characterization of the structure of finite words, analogous to the description of ultimately periodic infinite words obtained in the classical work of Morse and Hedlund. The proven inequality turns out to be useful in the study of finite-dimensional associative algebras. We now state the algebraic problem that we study using combinatorics on words.

Let $\A$ be an associative, unital, and finite-dimensional algebra over a field $\F$, $\SS \subseteq \A$. We consider products of elements of $\SS$ as words over an alphabet $\SS$. Note that we consider $1_\A$ as an empty word (word of the length 0). Let $\L(\SS)$ denote the unital algebra generated by $\SS$ (i.\,e., $\F$-linear span of  the set of all words over $\SS$). Let $\SS^{\leq i}$ and $\SS^i$ denote the set of all words (products) over $\SS$ of length at most $i$ and exactly $i$, respectively, here $i\geq 0$. By $\L_i(\SS)$ we denote $\langle \SS^{\leq i}\rangle$ ($\F$-linear span of $\SS^{\leq i}$).

The {\em length of a subset} $\SS$ is the smallest integer $k$ such that $\L(\SS)=\L_k(\SS)$. We denote the length of a subset $\SS$ by $l(\SS)$. We call $\SS$ a {\it generating set} of $\A$ if $\L(\SS)=\A.$

\begin{Definition}\label{alg_len}
The {\em length} of an $\F$-algebra $\A$ is the number $$l(\A)=\max \{l(\SS) \mid \text{$\SS$ is a generating set of $\A$}\}.$$
\end{Definition}

In \cite{Paz}, Paz proves $l(M_n(\F))\leq \lceil (n^2+2)/3 \rceil$ (where $\lceil \cdot \rceil $ denotes the least integer function and $M_n(\F)$ denotes the algebra of $n \times n$ matrices over $\F$). In \cite{Pap}, Pappacena proves an upper bound for $l(\A)$ as a function of $\dim\A$ and $m(\A)$ (where $m(\A)$ is the largest degree of the minimal polynomial among all elements of the algebra $\A$). 

In both of the above works, the authors used counting the number of linearly independent words over $\SS$ to obtain estimates. In this work we pay more attention to the combinatorics on words. This allows us to obtain an upper bound for $l(\A)$, which, in particular, generalizes the above and some other bounds.

The paper is organized as follows. In Section \ref{section:comb_on_words}, we introduce some notations from the combinatorics on words. In Section \ref{section:MH_fin}, we prove a finite version of the Morse--Hedlund Theorem. Namely, we describe the form of finite words that satisfy the condition $f_W(n)\leq n$ for some $n$ (where $f_W(n)$ is the number of different factors of length $n$ appearing in $W$) in Theorem \ref{MH_fin}. Section \ref{section:MH_fin_gen} contains a generalization of Theorem \ref{MH_fin}. Namely, we describe the form of finite words that satisfy the condition $f_W(n)\leq m$ for some $n \in [m,l-m]$, where $W$ is a word of length $l$. Section \ref{section:c(W)} is devoted to relations between power avoidance and subword complexity of finite words. In Section \ref{sec:length}, we present some estimates for the length of an algebra that are known to date. In Section \ref{sec:bound}, we prove an improved general upper bound for the length of an algebra using combinatorial result obtained in Section \ref{section:c(W)}. Namely, we prove that the length of an algebra does not exceed the maximum of $\dim \A / (k+1)+k-1$ and $k(m(\A)-1)$ for an arbitrary natural number $k$ in Theorem \ref{thm:main}. Section \ref{sec:compare} is devoted to comparison of the obtained bound with previously known bounds. Theorems \ref{MH_fin}, \ref{c(w)_Q} and \ref{thm:main} are the main results of this paper.

\section{Combinatorics on words}\label{section:comb_on_words}

A {\em word} is a finite (or right-infinite) string of elements called letters. A finite word $U$ is called a {\em factor} (resp., a {\em suffix}) of $W$ if there exist (possibly empty) words $V_1$, $V_2$ (resp., $V$) such that $W=V_1UV_2$ (resp., $W=VU$).

The {\em subword complexity} (also called {\em factor complexity} or just {\em complexity}) of a word $W$ is the function $f_W(n)$ that maps $n$ to the number of distinct factors of length $n$ in $W$. Note that $f_W(0)=1$ for all $W$, since there is an empty factor $\varepsilon$.

A factor $B$ of a word $W$ is called {\em (right) special} if there exist two occurrences of $B$ in $W$  which are followed on the right by two distinct letters.

A factor $B$ of a word $W$ is called {\em repeated in $W$} if there are at least two occurrences of $B$ in $W$. In the opposite case, the factor $B$ is called {\em unrepeated in $W$}.

An infinite word $W=a_1a_2\dots$ is called {\em ultimately periodic} if there exist $i,p \in \NN$ such that $a_j=a_{j+p}$ for all $j\geq i$.

In the combinatorics on words there is the following classical result of Morse and Hedlund, stated in \cite{MH}.

\begin{Theorem}[{\cite{MH}}]
    An infinite word $W$ is ultimately periodic if and only if there exists $n$ such that $f_W(n)\leq n$.
\end{Theorem}

This theorem has various generalizations. For example, a two-dimensional analogue of the Morse--Hedlund Theorem is known as Nivat's conjecture (see, e.\,g., \cite{Nivat}, \cite{CyrKra}).

The works on the combinatorics on words deal with infinite words (see \cite{Lot} for a general overview and a large bibliography) and finite words (see, e.\,g., \cite{DeLuca, AnJu}). Note that one should not think that the finite case is trivial compared to the infinite case. For example, the very useful property $f_W(n+1)\geq f_W(n)$, which is obvious for infinite words, is not true for finite words. The presence of a special factor of length $n$ in an infinite word obviously guarantees that $f_W(n+1) > f_W(n)$. This is also not true for a finite word. Since the proof of the original Morse--Hedlund Theorem relies on these trivial facts, its finite analogue requires a more subtle approach.

\section{A finite version of the Morse--Hedlund Theorem}\label{section:MH_fin}

In this section we describe the form of finite words that satisfy the condition $f_W(n)\leq n$ for some $n$.

Let $W=a_1\dots a_p$ be a finite word and $d \in \QQ$ be represented by the fraction $q/p$, where $q \geq 0$. Define the finite word $W^d=b_1\dots b_q$ such that $b_i=a_j$ for $i \equiv j \pmod p$. For example, $(abc)^{\frac{6}{3}}=(abc)^2=abcabc$, $(abc)^{\frac{1}{3}}=a$, $(abc)^{\frac{7}{3}}=abcabca$, $(abc)^0=\varepsilon$.

\begin{Lemma}\label{MH_fin_lem}
    Let $W=a_1\dots a_q(a_{q+1}\dots a_{q+p})^{\frac{l-q-t}{p}} a_{l-t+1}\dots a_{l}$ be a finite word. Then $f_W(i) \leq q+p+t$ for all $i\in \NN$.
\end{Lemma}

\begin{proof}
    Among the first $q$ factors of length $i$ of the word $W$ there are no more than $q$ distinct factors, among the last $t$ factors of length $i$ of the word $W$ there are no more than $t$ distinct factors, among the remaining factors of length $i$ of the word $W$ there are no more than $p$ distinct factors. Therefore, $f_W(i)\leq q+t+p$.
\end{proof}

\begin{Example}
    Let $W=abbabbabaa$. Then from Lemma \ref{MH_fin_lem} it follows that $f_W(i) \leq 5$ for all $i\in \NN$, since $W$ can be represented in the following form: $W=(abb)^{\frac{8}{3}}aa$ (i.\,e., $q=0$, $p=3$, $t=2$).
\end{Example}

If $2n>l$, then $f_W(n)\leq n$ holds for any word of length $l$. Therefore, we will further assume that $2n\leq l$.

\begin{Theorem}\label{MH_fin}
    Let $W=a_1\dots a_l$, $1\leq n\leq l/2$. Then $f_W(n)\leq n$ if and only if $W$ can be represented in the following form:     
    $$W=a_1\dots a_q(a_{q+1}\dots a_{q+p})^{\frac{l-q-t}{p}} a_{l-t+1}\dots a_{l},$$
    where $q+p+t \leq n$.
\end{Theorem}

\begin{proof}
    {\em Sufficiency}. Suppose $W=a_1\dots a_q(a_{q+1}\dots a_{q+p})^{\frac{l-q-t}{p}} a_{l-t+1}\dots a_{l},$ where $q+p+t \leq n$. From Lemma \ref{MH_fin_lem} it follows that $f_W(n)\leq q+t+p \leq n$.

    {\em Necessity}. Suppose $f_W(n)\leq n$. We prove the statement by induction  on $n \leq l/2$. The base case $n = 1$ is true, since $f_W(1) \leq 1$ guarantees that $W=a_1^l$ (i.\,e., $q=0,p = 1, t=0$). It is sufficient to deduce the statement for a fixed $n=k+1 \leq l/2$ from the assumption that it holds for $n=k$.

    Let $f_W(k+1)\leq k+1$. If $f_W(k)\leq k$ the statement follows from the induction hypothesis. Assume that $f_W(k) \geq k+1$. Since each factor of length $k$ of the word $W$ except the suffix of length $k$ of the word $W$ can be extended from the right to a factor of length $k+1$ of the word $W$, the word $W$ has at least $f_W(k) - 1$ distinct factors of length $k+1$ (i.\,e., $f_W(k+1) \geq f_W(k) - 1$). Therefore,
        \begin{align} \label{formula:k<k+1}
        k\leq f_W(k)-1 \leq f_W(k+1) \leq k+1.
        \end{align}
            
    Denote $r=f_W(k)$, $W_i=a_1\dots a_{l-i}$. If the suffix of length $k$ of the word $W_i$ is unrepeated in $W_i$ for all $i\in [0,l-k]$, then all factors of length $k$ of the word $W$ are unrepeated in $W$ and $f_W(k)=l-k+1$. Using \eqref{formula:k<k+1}, we obtain $l=f_W(k)+k-1\leq k+2+k-1=2k+1$. However, from the condition $l\geq 2n$ it follows that $l\geq 2k+2$. Obtained contradiction shows that the following number is well-defined.  
    $$t=\min\{i\in [0,l-k] \mid \text{the suffix of length $k$ of the word $W_i$ is repeated in $W_i$}\}.$$

    It follows from the definition of $t$ that for all $i<t$ the suffix of length $k$ of the word $W_i$ is unrepeated in $W_i$ and $f_{W_{i+1}}(k)=f_{W_{i}}(k)-1$. Therefore,
        \begin{align} \label{formula:W_t(k)}
        f_{W_t}(k)=f_{W}(k)-t=r-t.
        \end{align}

    Let us show that $W_i$ is a word of length at least $k+1$ for all $i<t$. The word $W_i$ is a word of length $l-i$. From \eqref{formula:k<k+1} it follows that $r\leq k+2$. From \eqref{formula:W_t(k)} it follows that $t\leq r$. From the condition $l\geq 2n$ it follows that $l\geq 2k+2$. Since $i<t$, we have
        \begin{align} \label{formula:l-i}
        l-i>l-t\geq l-r\geq 2k+2-(k+2)=k.
        \end{align}

    Note that since the suffix of length $k$ of the word $W_i$ is unrepeated in $W_i$ and $W_i$ is a word of length at least $k+1$, then the suffix of length $k+1$ of the word $W_i$ is unrepeated in $W_i$ and $f_{W_{i+1}}(k+1)=f_{W_{i}}(k+1)-1$. Therefore,
        \begin{align} \label{formula:W_t(k+1)}
        f_{W_t}(k+1)=f_{W}(k+1)-t.
        \end{align}
        
    We can choose $r-t$ distinct factors $C_1,\dots,C_{r-t}$ of length $k$ of the word $W_t$ such that $C_i$ is not the suffix of $W_t$ for all $i\in [1,r-t]$, by the definition of $t$. $C_1,\dots,C_{r-t}$ can be extended from the right to distinct factors of length $k+1$ of the word $W$. Therefore,
        \begin{align} \label{formula:W_t(k+1)>W_t(k)}
        f_{W_t}(k+1)\geq f_{W_t}(k).
        \end{align}

    Using \eqref{formula:k<k+1}, \eqref{formula:W_t(k)}, \eqref{formula:W_t(k+1)}, \eqref{formula:W_t(k+1)>W_t(k)}, we obtain
        \begin{align} \label{formula:k+1<k+1}
        k+1\leq f_W(k) = f_{W_t}(k)+t \leq f_{W_t}(k+1)+t= f_W(k+1) \leq k+1.
        \end{align}

    From \eqref{formula:k+1<k+1} it follows that $f_W(k+1)=f_W(k)=k+1=r$ and $f_{W_t}(k+1)=f_{W_t}(k)=k+1-t=r-t$. Hence, there are no special factors of $W_t$ among $C_1,\dots,C_{r-t}$ (i.\,e., $W_t$ has no special factors of length $k$). Indeed, otherwise $f_{W_t}(k+1) > f_{W_t}(k)$.

    From \eqref{formula:l-i} it follows that $l-t\geq k$. Then $W_t$ is a word of length at least $k$ and $f_{W_t}(k)\geq 1$. Therefore, from \eqref{formula:W_t(k)} it follows that $r-t\geq 1$. 
    
    Let $A_1,\dots, A_{r-t+1}$ be the first $r-t+1 \geq 2$ factors of length $k$ of the word $W_t$. Let us show that they all exist (i.\,e., the length of $W_t$ is at least $(r-t+1)+(k-1)$). The word $W_t$ is a word of length $l-t$ and from \eqref{formula:l-i} it follows that $l\geq r+k$. Then the length of $W_t$ is at least $r+k-t=(r-t+1)+(k-1)$.

    Since $f_{W_t}(k)=r-t$, there are $2$ equal factors among $A_1,\dots, A_{r-t+1}$. Then there are $q\geq 0$ and $p \geq 1$ such that $A_{q+1}=A_{q+1+p}$, where $q+1+p\leq r-t+1$. Then $q+p\leq r-t$ and $a_{q+1}=a_{q+1+p}, \dots, a_{q+k}=a_{q+k+p}$. Since $W_t$ does not contain special factors of length $k$, the letter following the factor $A_{q+1}$ coincides with the letter following the factor $A_{q+1+p}$ (i.\,e., $a_{q+k+1} = a_{q+k+p+1}$). Repeating this argument for the subsequent letters of the word $W_t$, we obtain $a_i = a_{i+p}$ for all $i \in [q+1, l-t-p]$.

    Therefore,
    $$W_t=a_1\dots a_q(a_{q+1}\dots a_{q+p})^{\frac{l-q-t}{p}},$$   
    $$W=a_1\dots a_q(a_{q+1}\dots a_{q+p})^{\frac{l-q-t}{p}} a_{l-t+1}\dots a_{l}$$    
    and $q+p+t \leq r-t+t = r = k+1$. The induction step is finished.

\end{proof}

\begin{Corollary}\label{MH_fin_cor}
    Let $W=a_1\dots a_l$, $n\leq l/2$, and $f_W(n)\leq n$. Then $f_W(i)\leq n$ for all $i\in \NN$.
\end{Corollary}
\begin{proof}
    From Theorem \ref{MH_fin} it follows that $W=a_1\dots a_q(a_{q+1}\dots a_{q+p})^{\frac{l-q-t}{p}} a_{l-t+1}\dots a_{l},$ where $q+p+t \leq n$. From Lemma \ref{MH_fin_lem} it follows that $f_W(i)\leq q+t+p \leq n$.
\end{proof}

\section{Generalization of Theorem \ref{MH_fin}}\label{section:MH_fin_gen}

We first translate with our notations a result of Levé and Séébold \cite{Leve}.

\begin{Lemma}\cite[Proposition 3.3]{Leve}\label{f_W}
    Let $W$ be a word of length $l$, and $m$ the least integer such that $f_W(m+1)\leq f_W(m)$. Then the function $f_W(n)$ is
\begin{itemize}
    \item strictly increasing for $n$ between $0$ and $m$;
    \item constant (equal to $f_W(m)$) for $n$ between $m$ and $l-f_W(m)+1$;
    \item decreasing by one for $n$ between $l-f_W(m)+1$ and $l$.
\end{itemize}
    
\end{Lemma}

The next theorem is a slight generalization of Theorem \ref{MH_fin}.

\begin{Theorem}\label{MH_fin_gen}
    Let $W=a_1\dots a_l$, $m \in \NN$, and $m\leq n\leq l-m$. Then $f_W(n)\leq m$ if and only if $W$ can be represented in the following form:     
    $$W=a_1\dots a_q(a_{q+1}\dots a_{q+p})^{\frac{l-q-t}{p}} a_{l-t+1}\dots a_{l},$$
    where $q+p+t \leq m$.
\end{Theorem}

\begin{proof}
    {\em Sufficiency}. Suppose $W=a_1\dots a_q(a_{q+1}\dots a_{q+p})^{\frac{l-q-t}{p}} a_{l-t+1}\dots a_{l},$ where $q+p+t \leq m$. From Lemma \ref{MH_fin_lem} it follows that $f_W(n)\leq q+t+p \leq m$.

     {\em Necessity}. Suppose $f_W(n)\leq m$. Our goal is to show that $f_W(m)\leq m$. Assume the contrary, i.\,e., $f_W(m) > m \geq f_W(n)$. Then from Lemma \ref{f_W} it follows that $f_W(n+i)\leq m-i$ for all $i \in [0,l-n]$. Hence $1=f_W(l)\leq m+n-l \leq 0$. A contradiction is obtained.

     Since $f_W(m)\leq m$ and $m\leq l/2$, from Theorem \ref{MH_fin} it follows that 
     $$W=a_1\dots a_q(a_{q+1}\dots a_{q+p})^{\frac{l-q-t}{p}} a_{l-t+1}\dots a_{l},$$
     where $q+p+t \leq m$.
\end{proof}

\begin{Corollary}\label{cor_gen}
    Let $W=a_1\dots a_l$, $m \in \NN$, $m\leq n\leq l-m$, and $f_W(n)\leq m$. Then $f_W(i)\leq m$ for all $i\in \NN$.
\end{Corollary}
\begin{proof}
    From Theorem \ref{MH_fin_gen} it follows that $W=a_1\dots a_q(a_{q+1}\dots a_{q+p})^{\frac{l-q-t}{p}} a_{l-t+1}\dots a_{l},$ where $q+p+t \leq m$. From Lemma \ref{MH_fin_lem} it follows that $f_W(i)\leq q+t+p \leq m$.
\end{proof}

Note that Theorem \ref{MH_fin} is a special case of Theorem \ref{MH_fin_gen} with $m=n$.

\section{Total complexity and power avoidance}\label{section:c(W)}

A word $W$ is a {\em $d$-power} if $W=U^d$ for a finite word $U$ and some $d\in \QQ$. A word $W$ {\em avoids} $d$-powers (or {\em $d$-power-free}) if it has no factors that are $\alpha$-powers for $\alpha \geq d \geq 1$. A word $W$ {\em avoids} $d^+$-powers (or {\em $d^+$-power-free}) if it has no factors that are $\alpha$-powers for $\alpha > d \geq 1$. For example, $W=abcdbcdef$ avoids $2^+$-powers, but does not avoid $2$-powers, since $W$ contains the factor $(bcd)^2$. Note that $W$ avoids $1^+$-powers if and only if $W$ consists of pairwise distinct letters.

Power avoidance, subword complexity and relations between them are major themes in combinatorics on words (see, e.\,g., \cite{ShSh} and references therein). In this section we consider relations between power avoidance and subword complexity of a finite word.


\begin{Lemma}\label{lem_0,k}
    Let $W$ be a word of length $l$, $d\in \QQ$, $W$ avoid $d^+$-powers, $k\in \NN$, $k \leq l/2$, $l>kd$. Then $f_W(n)\geq n+1$ for all $n \in [0,k]$.
\end{Lemma}
\begin{proof}
    Assume the contrary, i.\,e., $f_W(n) \leq n$ for some $n \in [0,k]$. Since $f_W(0)=1$, $1\leq n\leq k \leq l/2$. From Theorem \ref{MH_fin} it follows that
    $$W=a_1\dots a_q(a_{q+1}\dots a_{q+p})^{\frac{l-q-t}{p}} a_{l-t+1}\dots a_{l},$$
    where $q+p+t\leq n$. Since $W$ avoids $d^+$-powers, $(l-q-t)/p \leq d$. Therefore,     
    $$l-q-t \leq pd \leq (n-q-t)d,$$     
    $$l \leq nd-(q+t)(d-1)\leq nd \leq kd.$$
    However, $l>kd$. A contradiction is obtained.
\end{proof}

\begin{Lemma}\label{lem_k,l-k}
    Let $W$ be a word of length $l$, $d\in \QQ$, $W$ avoid $d^+$-powers, $k\in \NN$, $k \leq l/2$, $l>kd$. Then $f_W(n)\geq k+1$ for all $n \in [k,l-k]$.
\end{Lemma}
\begin{proof}
    Assume the contrary, i.\,e., $f_W(n) \leq k$ for some $n \in [k,l-k]$. Now we apply Corollary \ref{cor_gen} for $m=k$. We get $f_W(k) \leq k$. However, Lemma \ref{lem_0,k} guarantees that $f_W(k) \geq k+1$. A contradiction is obtained.   
\end{proof}

\begin{Lemma}\label{lem_l-k,l}
    Let $W$ be a word of length $l$, $d\in \QQ$, $W$ avoid $d^+$-powers, $k\in \NN$, $k \leq l/2$, $l>kd$. Then $f_W(n) = l-n+1$ for all $n \in [l-k,l]$.
\end{Lemma}
\begin{proof}
    It is obvious that $f_W(n) \leq l-n+1$ and $f_W(l) = 1$. Assume that $f_W(n) \leq l-n$ for some $n \in [l-k,l-1]$. Since $k \leq l/2$ and $n\geq l-k$, $l-n \leq k \leq n$. Now we apply Theorem \ref{MH_fin_gen} for $m=l-n$. We get
    $$W=a_1\dots a_q(a_{q+1}\dots a_{q+p})^{\frac{l-q-t}{p}} a_{l-t+1}\dots a_{l},$$
    where $q+p+t\leq l-n$. Since $W$ avoids $d^+$-powers, $(l-q-t)/p \leq d$. Therefore,     
    $$l-q-t \leq pd \leq (l-n-q-t)d,$$     
    $$l \leq (l-n)d-(q+t)(d-1)\leq (l-n)d \leq kd.$$
    However, $l>kd$. A contradiction is obtained. 
\end{proof}

The {\em total complexity} of a word $W=a_1\dots a_l$ is the number $c(W) = \sum_{i=0}^l f_W(i)$.

\begin{Theorem}\label{c(w)_Q}
    Let $W$ be a word of length $l$, $d\in \QQ$, $W$ avoid $d^+$-powers, $k\in \NN$, $k \leq l/2$, $l>kd$. Then $c(W)\geq (k+1)(l-k+1)$.
\end{Theorem}
\begin{proof}
    From Lemma \ref{lem_0,k} it follows that $\sum_{i=0}^{k-1} f_W(i) \geq k(k+1)/2.$

    From Lemma \ref{lem_k,l-k} it follows that $\sum_{i=k}^{l-k} f_W(i) \geq (k+1)(l-2k+1).$

    From Lemma \ref{lem_l-k,l} it follows that $\sum_{i=l-k+1}^{l} f_W(i) = k(k+1)/2.$

    Therefore, 
    $$c(W) \geq k(k+1)/2 + (k+1)(l-2k+1) + k(k+1)/2 = (k+1)(l-k+1).$$
\end{proof}

\begin{Corollary}\label{c(w)}
    Let $W$ be a word of length $l$, $d\in \NN$, $W$ avoid $d^+$-powers, $k\in \NN$, $l>kd$. Then $c(W)\geq (k+1)(l-k+1)$.
\end{Corollary}
\begin{proof}
    Suppose $d=1$. Then $W$ consists of pairwise distinct letters and $c(W) = l(l+1)/2+1$.
    On the other hand, $(k+1)(l-k+1)=-k^2+lk+l+1 \leq l(l+4)/4+1 \leq l(l+1)/2+1$, since $l \geq 2$. Thus, $c(W)\geq (k+1)(l-k+1)$.

    Suppose $d \geq 2$. Then $k < l/2$. From Theorem \ref{c(w)_Q} it follows that $c(W)\geq (k+1)(l-k+1)$.
    
\end{proof}

\begin{Example}
    Let $W=abbabbabbb$, $l=10$, $d=k=3$. Then from Theorem \ref{c(w)_Q} it follows that $c(W)\geq (3+1)(10-3+1)=32$. Below we list all the factors of $W$.
    $$\varepsilon, a, b, ab, ba, bb, abb, bab, bba, bbb, abba, abbb, babb, bbab,$$
    $$abbab, babba, babbb, bbabb, abbabb, babbab, bbabba, bbabbb, abbabba, abbabbb,$$
    $$babbabb, bbabbab, abbabbab, babbabbb, bbabbabb, abbabbabb, bbabbabbb, W.$$
    Therefore, $c(W)=32$.
\end{Example}

\section{Bounds for the length of an algebra}\label{sec:length}

The study of lengths of algebras began in 1984 with the study of the length of the algebra of $n \times n$ matrices $M_n(\FF)$ in a paper \cite{Paz} where the following estimate was proved.

\begin{Proposition}\label{prop:paz}\cite[Theorem 1]{Paz}
     Let $\F$ be an arbitrary field. Then 
     $$l(M_n(\FF)) \leq \left\lceil\frac{n^2 + 2}3 \right\rceil.$$
\end{Proposition}

In the same work it was conjectured that $l(M_n(\FF))=2n-2$. This hypothesis is known as Paz’s conjecture.

In recent years, the values of the lengths of various matrix algebras and group algebras have been actively studied (see, e.\,g., \cite{Kh24,Kh25} and references therein).

Note that for a set $\SS$ its length evaluation is a standard linear algebraic problem of constructing a certain basis of $\L(\SS)$. However, in the definition of the length of the algebra $\A$ we consider the set of {\em all} generating sets of $\A$. This explains the complexity of length computations for algebras.

\begin{Example}\label{ex:M_2}
    Let $\A$ be the algebra of $2 \times 2$ matrices over a field $\FF$ and $\SS = \{a,b\}$, where
    $$a=\begin{pmatrix}
        0 & 1 \\0 & 0
    \end{pmatrix},\quad
    b=\begin{pmatrix}
        0 & 0 \\1 & 0
    \end{pmatrix}.$$
    For $X \subseteq \A$, by $\langle X \rangle$ we denote the $\F$-linear span of $X$. It is obvious that $\L(\SS) \neq \L_1(\SS)= \langle \varepsilon, a, b\rangle$ and $\L(\SS) = \L_2(\SS)= \langle \varepsilon, a, b, ab\rangle$ (in this case $\varepsilon$ is an identity matrix). Then $l(\SS)=2$. Since $\L(\SS)=\A$, $l(\A)\geq 2$.
    
\end{Example}

As we see in Example \ref{ex:M_2}, a lower bound for the length of an algebra can be obtained by calculating the length of a particular generator system. To obtain upper bounds, general observations are required. Let us consider some general upper bounds for the length that are known to date.

The following upper bound is trivial (see, e.\,g., \cite[Remark 1.2]{Kh25}).

\begin{Proposition}\label{prop:triv}
     Let $\A$ be an associative finite-dimensional unital algebra. Then
$$l(\A)\leq \dim \A -1.$$
\end{Proposition}

For $a\in \A$, we denote by $\mu_a(t)\in \F[t]$ and $\deg a$ the minimal polynomial of the element $a$ over $\F$ and its degree, respectively. Since $\A$ is finite dimensional, it follows that  $\deg a \leq \dim\A$. Then we define $m(\A)=\max\{\deg a \mid a\in \A\}$ and $m(\SS^*)=\max\{\deg W \mid  \text{$W\in \SS^i$ for some $i\geq 0$}\}$. For example, $m(M_n(\F))=n$, by the Cayley--Hamilton theorem. The case $m(\A)=1$ is trivial. Indeed, in this case $\A=\langle 1_{\A} \rangle$. Therefore, we further assume that $m(\A)>1$.

The following upper bound was proved in \cite{dih} for calculating the length of the group algebra of the dihedral group.

\begin{Proposition}\label{prop:kh}\cite[Theorem 3.5]{dih}
     Let $\A$ be an associative finite-dimensional unital algebra. Then
$$l(\A)\leq \max\left\{m(\A)-1,\frac{\dim\A}{2}\right\}.$$
\end{Proposition}

The following general upper bound in the particular case $\A=M_n(\F)$ asymptotically improves bound on $l(M_n(\FF))$ from Proposition \ref{prop:paz}.

\begin{Proposition}\label{prop:pap}\cite[Corollary 3.2]{Pap}
     Let $\A$ be an associative finite-dimensional unital algebra, $m=m(\A)>1$, $d=\dim \A$. Then
$$l(\A) < m \sqrt{\dfrac{2d}{m-1} + \dfrac{1}{4}} + \dfrac{m}{2} - 2.$$
\end{Proposition}

Despite the fact that the bound from Proposition \ref{prop:pap} is asymptotically better than the others in the case $\A = M_n(\F)$, there are no examples for which this estimate would be sharp. On the other hand, bounds from Proposition \ref{prop:paz}, Proposition \ref{prop:triv} and Proposition \ref{prop:kh} are sharp for $n\leq 4$, one generated algebras and group algebra of the dihedral group, respectively. In the next section we prove an improved version of the general upper bound on the length using combinatorial results from Section \ref{section:c(W)}. In particular, obtained bound is sharp in the cases described above.

\section{An improved general upper bound for the length}\label{sec:bound}

The main result of this section, Theorem \ref{thm:main}, generalizes all the above bounds for $l(\A)$ and contains them as special cases.

\begin{Remark} \label{rem:shift}
a) For $\SS, \widetilde \SS \subseteq \A$, we notice that $l(\SS) = l(\widetilde\SS)$ whenever $\langle \SS \cup \{1_\A\} \rangle = \langle \widetilde\SS  \cup \{1_\A\} \rangle$. Particularly, we may shift each element $x \in \SS$, replacing it with $x + \lambda_x \cdot 1_\A$ for any $\lambda_x \in \FF$, and the length of the obtained set will be the same.

b) We have $l(\SS) = l(\langle\SS \rangle)$. This usually allows us to assume that $\SS$ is finite when computing $l(\SS)$ (by considering a basis of $\langle\SS \rangle$).

\end{Remark}


\begin{Proposition}\label{prop:shift}
    Let $\A$ be an associative finite-dimensional unital $\F$-algebra, $x\in \A$, $\deg x = m$, $|\F|>m$. Then there exists $\lambda \in \F$ such that $x+\lambda \cdot 1_{\A}$ is invertible and $(x+\lambda \cdot 1_{\A})^{-1} \in \L_{m-1}(\{x\})$.
\end{Proposition}
\begin{proof}
    Clearly, $\L(\{x\})$ is $m$-dimensional subalgebra of the algebra $\A$ and 
    $$\L(\{x\})=\L_{m-1}(\{x\}) \cong \F[t]/(\mu_x(t)).$$ 
    Then $x+\lambda\cdot 1_{\A}$ is invertible if and only if $t+\lambda$ and $\mu_x(t)$ are coprime. Since $|\F|>m=\deg \mu_x(t)$, there exists $\lambda$ such that $t+\lambda$ and $\mu_x(t)$ are coprime. On the other hand, $(x+\lambda \cdot 1_{\A})^{-1} \in \L(\{x\})=\L_{m-1}(\{x\})$.
\end{proof}

\begin{Definition}  A word $W \in  \SS^{j}$ of length $j$ is called  {\em reducible over $\SS$\/} if there is  a number $i<j$ such that  $W \in \L_i(\SS)$ (i.\,e., $W$ can be expressed as a linear combination of  words of smaller lengths). If $W$ is not reducible, then it is called
{\em irreducible over  $\SS$\/}.
\end{Definition}

It is clear that if $W$ contains a reducible factor, then $W$ is reducible. Note that irreducible words of length $j$ over $\SS$ exist if and only if $j\leq l(\SS)$.

\begin{Lemma}\label{lem:reduce_power}
    Let $\A$ be an associative finite-dimensional $\F$-algebra, $|\F|>m(\A)$, $\SS \subseteq \A$. Then any irreducible over $\SS$ word $W$ avoids $(m(\A)-1)^+$-powers.
\end{Lemma}
\begin{proof}
    Suppose $W$ has a factor $U^d$, where $U=a_1\dots a_k$, $a_i\in \SS$ for all $i\in [1,k]$. By the definition of $m(\A)$,
    $$U^{m(\A)} =\sum_{j=0}^{m(\A)-1} \lambda_j U^j,$$
    where $\lambda_0, \dots \lambda_{m(\A)-1} \in \F$. Since $|\F|>m(\A)$, by Remark \ref{rem:shift} and Proposition \ref{prop:shift} we may assume that $a_1,\dots, a_k$ are invertible. Multiplying both sides of the equality by  $a_k^{-1}\dots a_2^{-1}$ from the right, we get
    $$U^{\frac{k(m(\A)-1)+1}{k}} =\sum_{j=1}^{m(\A)-1} \lambda_j U^{j-1}a_1 + \lambda_0 a_k^{-1}\dots a_2^{-1}.$$
    
        Clearly, $\sum_{j=1}^{m(\A)-1} \lambda_j U^{j-1}a_1\in \L_{k(m(\A)-2)+1}(\SS)$. Since $\deg a_i \leq m(\A)$, Proposition \ref{prop:shift} guarantees that $a_i^{-1} \in \L_{\deg a_i-1}(\{a_i\})\subseteq \L_{m(\A)-1}(\{a_i\})\subseteq \L_{m(\A)-1}(\SS)$ for all $i\in [1,k]$. Therefore, $a_k^{-1}\dots a_2^{-1} \in \L_{(k-1)(m(\A)-1)}(\SS)$. Since $k(m(\A)-1)+1>k(m(\A)-2)+1$ and $k(m(\A)-1)+1>(k-1)(m(\A)-1)$, $U^{\frac{k(m(\A)-1)+1}{k}}$ is reducible.
        
        We obtain that $d<\frac{k(m(\A)-1)+1}{k}$, otherwise $W$ is reducible. Then $kd<k(m(\A)-1)+1$ and $d\leq (m(\A)-1)$. Thus, $W$ avoids $(m(\A)-1)^+$-powers.
\end{proof}

To prove the main result of this section we need to fix an order on words. For a finite set $\SS = \{a_1, a_2, \dots, a_k\}$, we define a total order on $\SS^*$ (finite-length words over $\SS$), called the {\it shortlex order}. First we order the letters from $\SS$, $a_1 \prec a_2 \prec \ldots \prec a_k$. Then for $U =a_{i_1}a_{i_2}\dots a_{i_p}$ and $W =a_{j_1}a_{j_2}\dots a_{j_q}$ from $\SS^*$, we set $U \prec W$ if and only if
\[
\text{$p < q$ or ($p = q$ and $i_1 = j_1, i_2 = j_2, \dots, i_{t-1} = j_{t-1}$ but $i_t < j_t$, for some $t \in [1, p]$)}.
\]
By the definition we see that $\prec$ is a total ordering of $\SS^*$.

\begin{Definition}
Let $\SS$ be a finite subset of an $\F$-algebra $\A$ and $\prec$ be a shortlex order on $\SS^*$. Denote by $\LIW{i}$ the least element of the set
$
\{U \in S^i \mid \text{$U$ is irreducible over $\SS$} \}
$ with respect to $\prec$. In other words, $\LIW{i}$ is the lexicographically minimal word among all the irreducible words over $\SS$ of length $i$.
\end{Definition}

Note that $\LIW{i}$ is irreducible by the definition. Therefore, $\LIW{i}$ exists if and only if $i \leq l(\SS)$.

The following lemma connects the combinatorial concept of word complexity with the algebraic concept of the dimension of an algebra.

\begin{Lemma}\label{s_W}\cite[Lemma 2.14]{Kh25}
    Let $\SS$ be a finite subset of an $\F$-algebra $\A$ and $\SS^*$ be equipped with a shortlex order $\prec$.
    Then $c(\LIW{i}) \leq \dim \A$, for any $i \geq 1$.
\end{Lemma}

\begin{Theorem}\label{thm:main}
     Let $\A$ be an associative finite-dimensional unital $\F$-algebra, $d=\dim\A$, $m=m(\A)>1$, $|\F|>m$. Then for any $k\in \NN$,
$$l(\A)\leq \max\left\{k(m-1),\frac{d}{k+1}+k-1\right\}.$$
\end{Theorem}
\begin{proof}
    Let $\SS$ be an arbitrary generating set of $\A$. We are to prove that for any $k\in \NN$, $l=l(\SS)\leq \max \left\{k(m-1),\frac{d}{k+1}+k-1\right\}$.

    Suppose $l > k(m-1)$. From Lemma \ref{lem:reduce_power} it follows that $\LIW{l}$ avoids $(m-1)^+$-powers. Now we apply Corollary \ref{c(w)} (for $W=\LIW{l}$ and $d=m-1$) and obtain $c(\LIW{l})\geq (k+1)(l-k+1)$. On the other hand, from Lemma \ref{s_W} it follows that $c(\LIW{l}) \leq d$. This completes the proof.
\end{proof}

\begin{Corollary}
      Let $\SS$ be a finite subset of an associative finite-dimensional unital $\F$-algebra $\A$, $d=\dim \L(\SS)$, $m=m(\SS^*)>1$, $|\F|>m$. Then for any $k\in \NN$,
$$l(\SS)\leq \max\left\{k(m-1),\frac{d}{k+1}+k-1\right\}.$$    
\end{Corollary}
\begin{proof}
    We can argue exactly as above replacing $m(\A)$ and $\dim \A$ with $m(\SS^*)$ and $\dim \L(\SS)$, respectively.
\end{proof}

\section{Comparison of Theorem \ref{thm:main} with other bounds}\label{sec:compare}

In this section we compare Theorem \ref{thm:main} with the currently known bounds for the length given in Section \ref{sec:length} (below we assume that $\F$ is sufficiently large). Obviously, Proposition \ref{prop:triv} and Proposition \ref{prop:kh} are special cases of Theorem \ref{thm:main} with $k=0$ and $k=1$, respectively.

Applying Theorem \ref{thm:main} to $\A=M_n(\F)$ with $k=2$, we obtain $l(M_n(\F))\leq \max \{2n-2,n^2/3 + 1\}=n^2/3+1$. Since $l(M_n(\F)) \in \ZZ$, $l(M_n(\F))\leq \lfloor n^2/3+1 \rfloor$. However, $\lfloor n^2/3+1 \rfloor = \lceil(n^2+2)/3 \rceil$. Thus, Proposition \ref{prop:paz} immediately follows from Theorem \ref{thm:main} with $k=2$.

Proposition \ref{prop:pap} is the closest result to Theorem \ref{thm:main}. It has the same asymptotic behavior $O(\sqrt{dm})$. Let us show that the bound from Theorem \ref{thm:main} is sharper. Now we apply Theorem \ref{thm:main} (for $k=\lfloor\sqrt{d/m} \rfloor$) and obtain

$$l(\A)\leq \max\left\{\left\lfloor\sqrt{d/m}\right\rfloor(m-1),\frac{d}{\left\lfloor\sqrt{d/m}\right\rfloor+1}+\left\lfloor\sqrt{d/m}\right\rfloor-1\right\} <$$
$$\max\left\{\sqrt{d/m}(m-1),\frac{d}{\sqrt{d/m}}+\sqrt{d/m}-1\right\}=$$
$$\max\left\{\sqrt{dm}-\sqrt{d/m},\sqrt{dm}+\sqrt{d/m}-1\right\}=\sqrt{dm}+\sqrt{d/m}-1=\sqrt{dm}(1+1/m)-1.$$

On the other hand,

$$ m \sqrt{\dfrac{2d}{m-1} + \dfrac{1}{4}} + \dfrac{m}{2} - 2 =  \sqrt{dm}\sqrt{\dfrac{2m}{m-1} + \dfrac{m}{4d}} + \dfrac{m}{2} - 2>\sqrt{dm}\sqrt{\dfrac{2m}{m-1}}-1.$$

Since $2m/(m-1)>(1+1/m)^2$,  Proposition \ref{prop:pap} immediately follows from Theorem \ref{thm:main} with $k=\lfloor\sqrt{d/m}\rfloor$.

Note that in some special cases significantly sharper bounds are known than those obtained in Theorem \ref{thm:main}. For example, in the case of commutative algebras (see \cite{Mar} and \cite{pgr}). However, Theorem \ref{thm:main} appears to be the best known result in general case.

\noindent {\bf Funding} 

The work is supported by the HSE University Basic Research Program.

\end{document}